%% file: RationalityF4example.tex
\documentclass{article}

\usepackage{amssymb, amsmath, amsthm, tikz, verbatim, graphicx,lmodern,indentfirst,enumerate,color}

\usepackage[titletoc,page]{appendix}
\usepackage[margin=1.4in]{geometry}

\newtheorem{thm}{Theorem}[section]

\newtheorem{prop}[thm]{Proposition}

\newtheorem{con}[thm]{Conjecture}
\newtheorem{oprob}[thm]{Open Problem}

\theoremstyle{definition}

\theoremstyle{definition}
\newtheorem{question}[thm]{Question}

\theoremstyle{definition}
\newtheorem{defn}[thm]{Definition}

\theoremstyle{remark}
\newtheorem{rem}[thm]{Remark}

\numberwithin{equation}{section}

\newcommand*\circled[1]{\tikz[baseline=(char.base)]{
            \node[shape=circle,draw,inner sep=0pt,minimum size=5mm] (char) {#1};}}

\newcommand{\rootsF}[5]{\circled{#1}
         	\begin{tabular}{cccc}
            #2&#3&#4&#5
            \end{tabular}}
            
\makeatletter
\newcommand{\rmnum}[1]{\romannumeral #1}
\newcommand{\Rmnum}[1]{\expandafter\@slowromancap\romannumeral #1@}
\makeatother
\begin{document}

\input{title}
\input{Introduction}

\input{Preliminaries}
\input{Pseudo-red}
\input{G-cr}

\input{TCC}
\input{acknowledgements}
\bibliography{mybib}

\end{document}

%% file: title.tex
\title{Complete reducibility of subgroups of reductive algebraic groups over nonperfect fields \Rmnum{4}: An $F_4$ example}
\author{Falk Bannuscher, Alastair Litterick, Tomohiro Uchiyama}
\date{}
\maketitle 

\begin{abstract}
Let $k$ be a nonperfect separably closed field. Let $G$ be a connected reductive algebraic group defined over $k$. We study rationality problems for Serre's notion of complete reducibility of subgroups of $G$. In particular, we present the first example of a connected nonabelian $k$-subgroup $H$ of $G$ that is $G$-completely reducible but not $G$-completely reducible over $k$, and the first example of a connected nonabelian $k$-subgroup $H'$ of $G$ that is $G$-completely reducible over $k$ but not $G$-completely reducible. This is new: all previously known such examples are for finite (or non-connected) $H$ and $H'$ only. 
\end{abstract}


%% file: Introduction.tex
\section{Introduction}
Let $k$ be a field. Let $\overline k$ be an algebraic closure of $k$. Let $G$ be a connected affine algebraic $k$-group: we regard $G$ as a $\overline k$-defined algebraic group together with a choice of $k$-structure in the sense of Borel~\cite[AG.~11]{Borel-AG-book}. We say that $G$ is \emph{reductive} if the unipotent radical $R_u(G)$ of $G$ is trivial. Throughout, $G$ is always a connected reductive $k$-group (unless stated otherwise). In this paper, we continue the investigation of rationality problems for complete reducibility of subgroups of $G$; see~\cite{Uchiyama-Nonperfect-pre},~\cite{Uchiyama-Nonperfectopenproblem-pre},~\cite{Uchiyama-Nonperfectopenproblem-pre-part3}. By a subgroup of $G$ we mean a (possibly non-$k$-defined) closed subgroup of $G$. Following Serre~\cite[Sec.~3]{Serre-building} we make the following definition.
\begin{defn}
A subgroup $H$ of $G$ is called \emph{$G$-completely reducible over $k$} (\emph{$G$-cr over $k$} for short) if whenever $H$ is contained in a $k$-defined parabolic subgroup $P$ of $G$, then $H$ is contained in a $k$-defined Levi subgroup of $P$. In particular, if $H$ is not contained in any proper $k$-defined parabolic subgroup of $G$, $H$ is called \emph{$G$-irreducible over $k$} (\emph{$G$-ir over $k$} for short). 
\end{defn}

This notion of $G$-complete reducibility faithfully generalises the notion of completely reducible representations, see~\cite[Sec.~3]{Serre-building} for more detail. So far, most studies on $G$-complete reducibility consider the case $k = \bar k$, for example~\cite{Liebeck-Seitz-memoir},~\cite{Stewart-nonGcr},~\cite{Litterick-Thomas-goodChar-Trans}, and not much is known for a non-algebraically closed $k$ (in particular for nonperfect $k$) except a few general results and important examples, see~\cite{Bate-cocharacter-Arx},~\cite[Sec.~5]{Bate-geometric-Inventione},~\cite{Uchiyama-Nonperfect-pre},~\cite{Uchiyama-Nonperfectopenproblem-pre},~\cite{Uchiyama-Nonperfectopenproblem-pre-part3}. We say that a subgroup $H$ of $G$ is $G$-cr ($G$-ir) if $H$ is $G$-cr ($G$-ir) over $\bar k$. Now it is natural to ask:

\begin{question}\label{MainQuestion}
Let $H$ be a subgroup of $G$. 
\begin{enumerate}
\item{If $H$ is $G$-cr, then is it $G$-cr over $k$?}
\item{If $H$ is $G$-cr over $k$, then is it $G$-cr?} 
\end{enumerate}
\end{question}

Here is the main result of this paper:
\begin{thm}\label{F4example}
Let $k$ be a nonperfect separably closed field of characteristic $2$. Let $G$ be a simple $k$-group of type $F_4$. Then there exists a connected nonabelian $k$-subgroup $H$ of $G$ that is $G$-cr but not $G$-cr over $k$, and a connected nonabelian $k$-subgroup $H'$ of $G$ that is $G$-cr over $k$ but not $G$-cr.
\end{thm}

Several comments are in order. First, when we consider Question~\ref{MainQuestion}, we can assume $k=k_s$ (a separable closure of $k$) by the following result, which is~\cite[Thm.~1.1]{Bate-separable-Paris}:
\begin{prop}\label{separableFieldExt}
A subgroup $H$ of $G$ is $G$-cr over $k$ if and only if $H$ is $G$-cr over $k_s$. 
\end{prop}
In particular, Question~\ref{MainQuestion} has an affirmative answer if $k$ is perfect. Proposition~\ref{separableFieldExt} depends on the recently proved and deep \emph{centre conjecture of Tits} (see Conjecture~\ref{centerconjecture}) in spherical buildings~\cite{Serre-building},~\cite{Tits-colloq},~\cite{Weiss-center-Fourier}. The centre conjecture (theorem) has been used to study complete reducibility over $k$, see~\cite{Bate-cocharacterbuildings-Arx},~\cite{Uchiyama-Nonperfectopenproblem-pre}. In this paper, we prove Proposition~\ref{centralizerA} which is related to a rationality problem for the centre conjecture. We assume $k=k_s$ throughout.  

Second, the third author had already found several examples of a subgroup $H$ (or $H'$) satisfying Theorem~\ref{F4example} for $G$ of type $D_4$, $G_2$, $E_6$, and $E_7$ and for $k$ of characteristic $2$ in~\cite{Uchiyama-Separability-JAlgebra},~\cite{Uchiyama-Nonperfect-pre},~\cite{Uchiyama-Nonperfectopenproblem-D4}. We stress that Theorem~\ref{F4example} not only extends our collection of such examples, but it is new: we give the first \emph{connected} such subgroup. In fact, our subgroups $H$ and $H'$ are just slight modifications of a simple subgroup of type $G_2$ in $F_4$. We are surprised with this pathological example since connected (moreover simple) subgroups of $G$ usually behave nicely. 

Third, it is not hard to find examples of such behaviour if we allow $H$ (or $H'$) to be non-$k$-defined. To find a $k$-defined $H$, we have used \emph{nonseparability} of $H$ (or a part of $H'$) in $G$. In fact, combining~\cite[Thms.~1.5 and~9.3]{Bate-cocharacter-Arx} we have that if a $k$-subgroup $H$ of $G$ is separable in $G$ and $H$ is $G$-cr, then it is $G$-cr over $k$. Recall~\cite[Def.~1.5]{Uchiyama-Separability-JAlgebra}:
\begin{defn}
Let $H$ and $N$ be affine algebraic groups. Suppose that $H$ acts on $N$ by group automorphisms. The action of $H$ on $N$ is called \emph{separable} if $\textup{Lie}\;C_N(H) = \mathfrak{c}_{\textup{Lie} N}(H)$. 
\end{defn}

Note that the notion of a separable action is a slight generalisation of that of a separable subgroup~\cite[Def.~1.1]{Bate-separability-TransAMS}. See~\cite{Bate-separability-TransAMS} and~\cite{Herpel-smoothcentralizerl-trans} for more on separability. In our construction, it is crucial  that our $H$ (or a part of $H'$) acts nonseparably on the unipotent radical of some proper $k$-parabolic subgroup of $G$. It is known that if the characteristic $p$ of $k$ is very good for $G$, every subgroup of $G$ is separable~\cite[Thm.~1.2]{Bate-separability-TransAMS}. This suggests that we need to work with small $p$. Such proper nonseparable subgroups are hard to find. Only a handful of such examples are known~\cite[Sec.7]{Bate-separability-TransAMS},~\cite{Uchiyama-Separability-JAlgebra},~\cite{Uchiyama-Nonperfect-pre},~\cite{Uchiyama-Nonperfectopenproblem-D4} and all of them are finite subgroups. In our example, $H$ (or the nonseparable part of $H'$) is connected (and very close to $G_2$).

Fourth, our method to construct $H$ and $H'$ via group-theoretic arguments and geometric invariant theory is almost identical to the constructions in $D_4$, $G_2$, $E_6$, and $E_7$ examples mentioned above~\cite{Uchiyama-Separability-JAlgebra},~\cite{Uchiyama-Nonperfect-pre},~\cite{Uchiyama-Nonperfectopenproblem-D4}. Since the same method works for many examples (for finite and connected $H$ and $H'$) we believe that there should be some general phenomenon underlying these constructions (cf.~\cite[Thms.~1.2, 1.3]{Uchiyama-Nonperfectopenproblem-pre-part3}). 
\begin{oprob}
Suppose that there exists a $k$-subgroup $H$ of $G$ that acts nonseparably on the unipotent radical of some proper $k$-parabolic subgroup of $G$. 
\begin{enumerate}
\item{If $H$ is $G$-cr, can one always find some modification $H'$ of $H$ such that $H'$ remains a $k$-subgroup, $H'$ remains $G$-cr, but for which $H'$ is not $G$-cr over $k$?}
\item{If $H$ is $G$-cr over $k$, can one always find some modification $H''$ of $H$ such that $H''$ remains a $k$-subgroup, $H''$ remains $G$-cr over $k$, but for which $H''$ is not $G$-cr?}
\end{enumerate}
\end{oprob}

We have also answered the second part of Question~\ref{MainQuestion} using a different method (Weil restriction) and a different language (scheme-theoretic). First, recall~\cite[Def.~1.1.1]{Conrad-pred-book}
\begin{defn}
Let $k$ be a field. Let $G$ be a connected affine algebraic $k$-group. If the $k$-unipotent radical $R_{u,k}(G)$ is trivial, $G$ is called \emph{pseudo-reductive}.
\end{defn}
Weil restriction is a standard tool to construct non-reductive pseudo-reductive groups~\cite[Ex.~1.6.1]{Conrad-pred-book}. Using Weil restriction and some scheme-theoretic argument, we show that:
\begin{prop}\label{GcrandpseudoredGLpm}
Let $k$ be a nonperfect field of characteristic $p$. Then for each power $p^s$ ($s\in \mathbb{N}$) there exists a $k$-subgroup $H$ of $G=GL_{p^s}$ that is $G$-cr over $k$ but not $G$-cr.
\end{prop} 

The subgroup $H$ we find in proving Proposition~\ref{GcrandpseudoredGLpm} is abelian. We can find an abelian example (that is $G$-cr over $k$ but not $G$-cr) without using Weil restriction: take a \emph{$k$-anisotropic unipotent element} of $G$ for a finite example, and take its connected centraliser for a connected example, see~\cite[Rem.~5.3]{Uchiyama-Nonperfectopenproblem-pre}. Remember that a unipotent element is called \emph{$k$-anisotropic} if it is not contained in any proper $k$-parabolic subgroup of $G$. This means the classical construction of Borel-Tits~\cite[Thm.~2.5]{Borel-Tits-unipotent-invent} fails for nonperfect $k$.  

We included Proposition~\ref{GcrandpseudoredGLpm} in this paper since it shows a relation between rationality problems for complete reducibility and pseudo-reductivity. We think that this topic should be investigated further. For example, the answer to the following open problem is true if $k$ is perfect or $C_G(H)$ is pseudo-reductive, see~\cite[Sec.~6]{Uchiyama-Nonperfectopenproblem-pre} for more on this:
\begin{oprob}\label{centraliserProblem}
Let $k$ be a field. Suppose that a $k$-subgroup $H$ of $G$ is $G$-cr over $k$. then is $C_G(H)$ $G$-cr over $k$? 
\end{oprob}

Here is the structure of the paper. In Section 2, we will set our notion and recall some important results about complete reducibility and related result from geometric invariant theory (GIT for short). In Section 3, we give a short review of Weil-restriction and prove Proposition~\ref{GcrandpseudoredGLpm}. Then, in Section 4, we prove Theorem \ref{F4example}. Note that since our method is almost identical to that in~\cite{Uchiyama-Separability-JAlgebra},~\cite{Uchiyama-Nonperfect-pre},~\cite{Uchiyama-Nonperfectopenproblem-D4}, our proof is just a minimum skeleton. Finally, in Section 5, we discuss a rationality problem related to complete reducibility and the centre conjecture.

%% file: Preliminaries.tex
\section{Preliminaries}
Throughout, we denote by $k$ a separably closed field. Our references for algebraic groups are~\cite{Borel-AG-book},~\cite{Borel-Tits-Groupes-reductifs},~\cite{Conrad-pred-book},~\cite{Humphreys-book1}, and~\cite{Springer-book}. 

Let $H$ be a (possibly non-connected) affine algebraic group. We write $H^{\circ}$ for the identity component of $H$. We write $[H,H]$ for the derived group of $H$. A reductive group $G$ is called \emph{simple} as an algebraic group if $G$ is connected and all proper normal subgroups of $G$ are finite. We write $X_k(G)$ and $Y_k(G)$ for the set of $k$-characters and $k$-cocharacters of $G$ respectively. For $\overline k$-characters and $\overline k$-cocharacters of $G$ we simply say characters and cocharacters of $G$ and write $X(G)$ and $Y(G)$ respectively. 

Fix a maximal $k$-torus $T$ of $G$ (such a $T$ exists by~\cite[Cor.~18.8]{Borel-AG-book}). Then $T$ is split over $k$ since $k$ is separably closed. Let $\Psi(G,T)$ denote the set of roots of $G$ with respect to $T$. We sometimes write $\Psi(G)$ for $\Psi(G,T)$. Let $\zeta\in\Psi(G)$. We write $U_\zeta$ for the corresponding root subgroup of $G$. We define $G_\zeta := \langle U_\zeta, U_{-\zeta} \rangle$. Let $\zeta, \xi \in \Psi(G)$. Let $\xi^{\vee}$ be the coroot corresponding to $\xi$. Then $\zeta\circ\xi^{\vee}:\overline k^{*}\rightarrow \overline k^{*}$ is a $k$-homomorphism such that $(\zeta\circ\xi^{\vee})(a) = a^n$ for some $n\in\mathbb{Z}$.
Let $s_\xi$ denote the reflection corresponding to $\xi$ in the Weyl group of $G$. Each $s_\xi$ acts on the set of roots $\Psi(G)$ by the following formula~\cite[Lem.~7.1.8]{Springer-book}:
$
s_\xi\cdot\zeta = \zeta - \langle \zeta, \xi^{\vee} \rangle \xi. 
$
\noindent By \cite[Prop.~6.4.2, Lem.~7.2.1]{Carter-simple-book} we can choose $k$-homomorphisms $\epsilon_\zeta : \overline k \rightarrow U_\zeta$  so that 
$
n_\xi \epsilon_\zeta(a) n_\xi^{-1}= \epsilon_{s_\xi\cdot\zeta}(\pm a)
            \text{ where } n_\xi = \epsilon_\xi(1)\epsilon_{-\xi}(-1)\epsilon_{\xi}(1).  \label{n-action on group}
$

The next result~\cite[Prop.~1.12]{Uchiyama-Nonperfect-pre} shows complete reducibility behaves nicely under central isogenies. 
\begin{defn}
Let $G_1$ and $G_2$ be reductive $k$-groups. A $k$-isogeny $f:G_1\rightarrow G_2$ is \emph{central} if $\textup{ker}\,df_1$ is central in $\mathfrak{g_1}$ where $\textup{ker}\,df_1$ is the differential of $f$ at the identity of $G_1$ and $\mathfrak{g_1}$ is the Lie algebra of $G_1$. 
\end{defn}
\begin{prop}\label{isogeny}
Let $G_1$ and $G_2$ be reductive $k$-groups. Let $H_1$ and $H_2$ be subgroups of $G_1$ and $G_2$ respectively. Let $f:G_1 \rightarrow G_2$ be a central $k$-isogeny. 
\begin{enumerate}
\item{If $H_1$ is $G_1$-cr over $k$, then $f(H_1)$ is $G_2$-cr over $k$.}
\item{If $H_2$ is $G_2$-cr over $k$, then $f^{-1}(H_2)$ is $G_1$-cr over $k$.} 
\end{enumerate}
\end{prop}

The next result~\cite[Thm.~1.4]{Bate-cocharacterbuildings-Arx} is used repeatedly to reduce problems on $G$-complete reducibility to those on $L$-complete reducibility where $L$ is a Levi subgroup of $G$. 

\begin{prop}\label{G-cr-L-cr}
Suppose that a subgroup $H$ of $G$ is contained in a $k$-defined Levi subgroup $L$ of $G$. Then $H$ is $G$-cr over $k$ if and only if it is $L$-cr over $k$. 
\end{prop}

We recall characterisations of parabolic subgroups, Levi subgroups, and unipotent radicals in terms of cocharacters of $G$~\cite[Prop.~8.4.5]{Springer-book}. These characterisations are essential to translate results on complete reducibility into the language of GIT; see~\cite{Bate-geometric-Inventione},~\cite{Bate-uniform-TransAMS} for example. 

\begin{defn}
Let $X$ be a affine $k$-variety. Let $\phi : \overline k^*\rightarrow X$ be a $k$-morphism of affine $k$-varieties. We say that $\displaystyle\lim_{a\rightarrow 0}\phi(a)$ exists if there exists a $k$-morphism $\hat\phi:\overline k\rightarrow X$ (necessarily unique) whose restriction to $\overline k^{*}$ is $\phi$. If this limit exists, we set $\displaystyle\lim_{a\rightarrow 0}\phi(a) := \hat\phi(0)$.
\end{defn}

\begin{defn}\label{R-parabolic}
Let $\lambda\in Y(G)$. Define
$
P_\lambda := \{ g\in G \mid \displaystyle\lim_{a\rightarrow 0}\lambda(a)g\lambda(a)^{-1} \text{ exists}\}, $\\
$L_\lambda := \{ g\in G \mid \displaystyle\lim_{a\rightarrow 0}\lambda(a)g\lambda(a)^{-1} = g\}, \,
R_u(P_\lambda) := \{ g\in G \mid  \displaystyle\lim_{a\rightarrow0}\lambda(a)g\lambda(a)^{-1} = 1\}. 
$
\end{defn}
Then $P_\lambda$ is a parabolic subgroup of $G$, $L_\lambda$ is a Levi subgroup of $P_\lambda$, and $R_u(P_\lambda)$ is the unipotent radical of $P_\lambda$. If $\lambda$ is $k$-defined, $P_\lambda$, $L_\lambda$, and $R_u(P_\lambda)$ are $k$-defined~\cite[Sec.~2.1-2.3]{Richardson-conjugacy-Duke}. All $k$-defined parabolic subgroups and $k$-defined Levi subgroups of $G$ arise in this way since $k$ is separably closed. It is well known that $L_\lambda = C_G(\lambda(\overline k^*))$. Note that $k$-defined Levi subgroups of a $k$-defined parabolic subgroup $P$ of $G$ are $R_u(P)(k)$-conjugate~\cite[Lem.~2.5(\rmnum{3})]{Bate-uniform-TransAMS}. 

Recall the following geometric characterisation for complete reducibility via GIT~\cite{Bate-geometric-Inventione}. Suppose that a subgroup $H$ of $G$ is generated by $n$-tuple ${\bf h}=(h_1,\cdots, h_n)$ of elements of $G$ (or ${\bf h}$ is a generic tuple of $H$ in the sense of~\cite[Def.~9.2]{Bate-cocharacter-Arx}), and $G$ acts on $G^n$ by simultaneous conjugation. 
\begin{prop}\label{geometric}
A subgroup $H$ of $G$ is $G$-cr if and only if the $G$-orbit $G\cdot {\bf h}$ is closed in $G^n$. 
\end{prop}
Combining Proposition~\ref{geometric} and a recent result from GIT~\cite[Thm.~3.3]{Bate-uniform-TransAMS} we have
\begin{prop}\label{unipotentconjugate}
Let $H$ be a subgroup of $G$. Let $\lambda\in Y(G)$. Suppose that ${\bf h'}:=\lim_{a\rightarrow 0}\lambda(a)\cdot {\bf h}$ exists. If $H$ is $G$-cr, then ${\bf h'}$ is $R_u(P_\lambda)$-conjugate to ${\bf h}$. 
\end{prop}
We use a rational version of Proposition~\ref{unipotentconjugate}; see~\cite[Cor.~5.1]{Bate-cocharacter-Arx},~\cite[Thm.~9.3]{Bate-cocharacter-Arx}:

\begin{prop}\label{rationalonjugacy}
Let $H$ be a subgroup of $G$. Let $\lambda\in Y_k(G)$. Suppose that ${\bf h'}:=\lim_{a\rightarrow 0}\lambda(a)\cdot {\bf h}$ exists. If $H$ is $G$-cr over $k$, then ${\bf h'}$ is $R_u(P_\lambda)(k)$-conjugate to ${\bf h}$. 
\end{prop}   

%% file: Pseudo-red.tex
\section{Complete reducibility and pseudo-reductivity}

In this section, we use the language of schemes. Recall~\cite[Sec.~A.5]{Conrad-pred-book}: 
\begin{defn}
Let $B\to B'$ be a finite flat map of noetherian rings, and $X'$ a quasi-projective $B'$-scheme. The \emph{Weil restriction} $R_{B'/B}(X')$ is a $B$-scheme of finite type satisfying the universal property $$ R_{B'/B}(X')(A)=X'(B'\otimes_B A) $$ for all $B$-algebras $A$.
\end{defn}

Now using a special case of a Weil restriction~\cite[Ex.~1.6.1]{Conrad-pred-book} we have
\begin{prop}\label{prbutnonred}
	Let $k'/k$ be a finite purely inseparable field extension with $k'\neq k$ and $G'$ a non-trivial smooth connected reductive $k'$-group. Then $G=R_{k'/k}(G')$ is a pseudo-reductive and non-reductive $k$-group.
\end{prop}
Note that $G=R_{k'/k}(G')$ is smooth due to~\cite[Prop.~A.5.11(1)]{Conrad-pred-book}. We use the following standard result~\cite[Ex.~3.2.2(a)]{Serre-building}:
\begin{prop}\label{reptheory}
Let $H$ be a $k$-subgroup of $\text{GL}(V)$ for some finite dimensional $k$-vector space $V$. Then $H$ is $\text{GL}(V)$-cr over $k$ (resp. $\text{GL}(V)$-ir over $k$) if and only if $V$ is a semisimple (resp. irreducible) $kH$-module.
\end{prop}
Combining~Propositions~\ref{prbutnonred} and~\ref{reptheory} together with the argument in~\cite[Sec.~5.2.1]{Bate-Stewart-Gems} (we reproduce the argument of~\cite[Sec.~5.2.1]{Bate-Stewart-Gems} to make this paper self-contained), we find: 

\begin{proof}[Proof of Proposition \ref{GcrandpseudoredGLpm}]
Let $k'/k$ be a purely inseparable field extension of degree $s$ in characteristic $p$. Fix a faithful action of $\mathbb{G}_m(k')$ on $\mathbb{G}_a(k')$ (we regard $\mathbb{G}_m(k')$ and $\mathbb{G}_a(k')$ as $k'$-groups).
Then this action induces an action of $H=R_{k'/k}(\mathbb{G}_m)$ on the $[k':k]$-dimensional $k$-vector group $R_{k'/k}(\mathbb{G}_a)$, obtaining an embedding of $H=R_{k'/k}(\mathbb{G}_m)$ in $G=GL_{p^s}$. 

It is easy to see that there are exactly two $\mathbb{G}_m(k')$-orbits on $\mathbb{G}_a(k')$. Thus there are exactly two $H(k)=(R_{k'/k}(\mathbb{G}_m))(k)$-orbits on $G(k)=(R_{k'/k}(\mathbb{G}_a))(k)$. This shows that $G$ is an irreducible $kH$-module and then $H$ is $G$-ir over $k$ by Proposition~\ref{reptheory}. Note that $H$ is not reductive by Proposition~\ref{prbutnonred}, so it is not $G$-cr by~\cite[Prop.~4.1]{Serre-building}.
\end{proof}

%% file: G-cr.tex
\section{$G$-cr vs $G$-cr over $k$ (Proof of Theorem~\ref{F4example})}

Let $G$ be a simple algebraic group of type $F_4$ defined over a nonperfect field $k$ of characteristic $2$. Fix a maximal torus $T$ of $G$ and a Borel subgroup $B$ of $G$ containing $T$. Let $\Psi(G)=\Psi(G,T)$ be the set of roots corresponding to $T$, and $\Psi(G)^{+}=\Psi(G,B,T)$ be the set of positive roots of $G$ corresponding to $T$ and $B$. The following Dynkin diagram defines the set of simple roots of $G$.
\begin{figure}[!h]
	\centering
	\scalebox{0.7}{
		\begin{tikzpicture}
		\draw (1,0) to (2,0);
		\draw (2,0.1) to (3,0.1);
		\draw (2,-0.1) to (3,-0.1);
		\draw (3,0) to (4,0);
		\draw (2.4,0.2) to (2.6,0);
		\draw (2.4,-0.2) to (2.6,0);
		\fill (1,0) circle (1mm);
		\fill (2,0) circle (1mm);
		\fill (3,0) circle (1mm);
		\fill (4,0) circle (1mm);
		\draw[below] (1,-0.2) node{$\alpha$};
		\draw[below] (2,-0.2) node{$\beta$};
		\draw[below] (3,-0.2) node{$\gamma$};
		\draw[below] (4,-0.2) node{$\delta$};
		\end{tikzpicture}}
\end{figure}

We label $\Psi(G)^{+}$ as in the following. The corresponding negative roots are labelled accordingly. For example, the roots $1$, $2$, $3$, $17$ correspond to $\alpha$, $\beta$, $\gamma$, $\delta$ respectively.
\begin{table}[!h]
	\begin{center}
		\scalebox{0.7}{
			\begin{tabular}{cccccc}
				\rootsF{1}{1}{0}{0}{0}&\rootsF{2}{0}{1}{0}{0}&\rootsF{3}{0}{0}{1}{0}&\rootsF{4}{1}{1}{0}{0}&\rootsF{5}{0}{1}{1}{0}&\rootsF{6}{1}{1}{1}{0}\\
				\rootsF{7}{0}{1}{2}{0}&\rootsF{8}{1}{1}{2}{0}&\rootsF{9}{1}{2}{2}{0}&\rootsF{10}{0}{1}{2}{2}&\rootsF{11}{1}{1}{2}{2}&\rootsF{12}{1}{2}{2}{2}\\
				\rootsF{13}{1}{2}{3}{2}&\rootsF{14}{1}{2}{4}{2}&\rootsF{15}{1}{3}{4}{2}&\rootsF{16}{2}{3}{4}{2}&\rootsF{17}{0}{0}{0}{1}&\rootsF{18}{0}{0}{1}{1}\\
				\rootsF{19}{0}{1}{1}{1}&\rootsF{20}{1}{1}{1}{1}&\rootsF{21}{0}{1}{2}{1}&\rootsF{22}{1}{1}{2}{1}&\rootsF{23}{1}{2}{2}{1}&\rootsF{24}{1}{2}{3}{1}\\
			\end{tabular}
		}
	\end{center}
\end{table}

Let $\lambda=13^{\vee}= 2\alpha^{\vee}+4\beta^{\vee}+3\gamma^{\vee}+2\delta^{\vee}$. Then 
$P_\lambda=\langle T, U_{\zeta}\mid \zeta\in \Psi(G)^{+}\cup \{-1,\cdots,-9\} \rangle$, 
$L_\lambda=\langle T, U_{\zeta}\mid \zeta\in \{\pm 1,\cdots,\pm 9\} \rangle$, 
and $R_u(P_\lambda)=\langle U_{\zeta} \mid \zeta \in \Psi(G)^{+}\backslash \{1, \cdots, 9\} \rangle$. Note that $L_\lambda$ is of type $B_3$. In this section, we use the commutator relations~\cite[Lem.~32.5, Props.~33.3 and 33.4]{Humphreys-book1} repeatedly. Let
\begin{equation*}
	M=\langle \epsilon_{2}(x_1), \epsilon_{-2}(x_2), \epsilon_{1}(x_3)\epsilon_{3}(x_3), \epsilon_{-1}(x_4)\epsilon_{-3}(x_4)\mid x_i\in \overline{k}\rangle.
\end{equation*}
Using the commutator formulae for root elements of $G$ (we used those found in \textsf{Magma}~\cite{Magma}), it is straightforward to check that the above generators for $M$ satisfy commutator relations, so that $M$ is a simple algebraic group of type $G_2$ with the given generators as root elements, cf.~\cite[pp.~72--77]{Carter-simple-book}. Let $a\in k\backslash k^2$ and $v(\sqrt a) = \epsilon_{20}(\sqrt a)\epsilon_{21}(\sqrt a)$. In the rest of the paper, the dot action always represents simultaneous conjugation. Define
\begin{equation*}
H := v(\sqrt a)\cdot M=\langle \epsilon_{2}(x_1), \epsilon_{-2}(x_2), \epsilon_{1}(x_3)\epsilon_{3}(x_3)\epsilon_{14}(ax_3), \epsilon_{-1}(x_4)\epsilon_{-3}(x_4)\epsilon_{12}(ax_4)\mid x_i\in \overline{k}\rangle.
\end{equation*}
The first main result in this section is
\begin{prop}\label{FirstMain}
$H$ is connected, $k$-defined, and $G$-cr, but not $G$-cr over $k$.
\end{prop}
\begin{proof}
It is clear that $H$ is generated by connected subgroups, each of which is defined over $k$, so, $H$ is connected and $k$-defined by~\cite[AG.~11]{Borel-AG-book}. Also $H$ is $G$-cr since $M$ is $L_\lambda$-ir by~\cite[Table~10, ID~3]{Litterick-Thomas-ex-groups}. We show that $H$ is not $G$-cr over $k$. Suppose the contrary. Choose $b\in k$ with $b^3=1$ and $b\neq 1$. Let $\bold{h}$ be a generic tuple of $H$ containing $\beta^{\vee}(b)$ and $\epsilon_{1}(1)\epsilon_{3}(1)\epsilon_{14}(a)$. Then $\bold{h'} := \lim_{a\rightarrow 0}\lambda(a)\cdot \bold{h}$ exists since $H<P_\lambda$. Then by Proposition~\ref{rationalonjugacy}, $\bold{h'}$ must be $R_u(P_\lambda)(k)$-conjugate to $\bold{h}$. Let $\bold{v} = (\beta^{\vee}(b), \epsilon_{1}(1)\epsilon_{3}(1)\epsilon_{14}(a))$. Then $\bold{v'}:=\lim_{a\rightarrow 0}\lambda(a)\cdot\bold{v}=(\beta^{\vee}(b), \epsilon_{1}(1)\epsilon_{3}(1))$. Thus there exists an element $u\in R_u(P_\lambda)(k)$ with $\bold{v} = u\cdot \bold{v'}$, which implies that $u$ commutes with $\beta^{\vee}(b)$. By~\cite[Prop.~8.2.1]{Springer-book} we can set $u=\prod_{i=10}^{24}\epsilon_{i}(x_i)$ for some $x_i\in k$. Then $x_{i}=0$ for $i\in\{11,12,14,15,18,19,22,23\}$. The equation $\bold{v} = u\cdot \bold{v'}$ also implies
\begin{align}\label{importantEq}
\epsilon_{1}(1)\epsilon_{3}(1)\epsilon_{14}(a) &= (\epsilon_{10}(x_{10})\epsilon_{13}(x_{13})\epsilon_{16}(x_{16})\epsilon_{17}(x_{17})\epsilon_{20}(x_{20})\epsilon_{21}(x_{21})\epsilon_{24}(x_{24}))\cdot (\epsilon_{1}(1)\epsilon_{3}(1))\nonumber \\
&=\epsilon_{1}(1)\epsilon_{3}(1)\epsilon_{11}(x_{10})\epsilon_{14}(x_{21}^2)\epsilon_{18}(x_{17})\epsilon_{22}(x_{20}+x_{21}).
\end{align}
This means $a=x_{21}^2$, which is a contradiction. 
\end{proof}

\begin{rem}
From the calculation above, we see that the curve $C(x):=\{\epsilon_{20}(x)\epsilon_{21}(x)\}$ is not contained $C_{R_u(P_\lambda)}(M)$, but the corresponding element in $\textup{Lie}(R_u(P_\lambda))$, that is, $e_{20}+e_{21}$ is in $\mathfrak{c}_{\textup{Lie}(R_u(P_\lambda))}(M)$. Then the argument in the proof of~\cite[Prop.~3.3]{Uchiyama-Separability-JAlgebra} shows that $M$ (hence $H$) acts nonseparably on $R_u(P_\lambda)$. 
\end{rem}

We move on to the second main result in this section. We use the same $k$, $a$, $b$, $G$, $M$, and $\lambda$ as above. Let $v(\sqrt a)=\epsilon_{-20}(\sqrt a)\epsilon_{-21}(\sqrt a)$. Define
\begin{equation*}
H' := \langle v(\sqrt a)\cdot M, \epsilon_{18}(x) \mid x\in \bar k \rangle.
\end{equation*}
\begin{prop}\label{SecondMain}
$H'$ is connected, $k$-defined, and $G$-cr over $k$, but not $G$-cr. 
\end{prop}
\begin{proof}
We have
\begin{equation*}
v(\sqrt a)\cdot M = \langle \epsilon_{2}(x_1), \epsilon_{-2}(x_2), \epsilon_{1}(x_3)\epsilon_{3}(x_3)\epsilon_{-12}(ax_3), \epsilon_{-1}(x_4)\epsilon_{-3}(x_4)\epsilon_{-14}(ax_4)\mid x_i\in \bar k\rangle.
\end{equation*}
Since $H'$ is generated by $k$-defined connected subgroups of $G$, it is connected and $k$-defined by~\cite[AG.~11]{Borel-AG-book}. In the following, we use the nonseparability of $M$ in $G$ again. We show that $H'$ is not $G$-cr. From~\cite[Table 38, $X=G_2$, ID 3, $p=2$]{Litterick-Thomas-ex-groups} $C_G(M)^{\circ}$ is of type $A_1$. By a direct computation using commutator relations, we have $G_{13}\leq C_G(M)^{\circ}$. Then $C_G(M)^{\circ}=G_{13}$. Now an easy calculation gives $C_G(H')^{\circ}=v(\sqrt a)\cdot U_{13}$. (The point is that $U_{-13}$ is not centralised by $U_{18}$.) Thus $C_G(H')^{\circ}$ is unipotent. Then by the classical result of Borel and Tits~\cite[Prop.~3.1]{Borel-Tits-unipotent-invent}, $C_G(H')^{\circ}$ is not $G$-cr. Since $C_G(H')^{\circ}$ is a normal subgroup of $C_G(H')$, by~\cite[Ex.~5.20]{Bate-uniform-TransAMS}, $C_G(H')$ is not $G$-cr. Then $H'$ is not $G$-cr by~\cite[Cor.~3.17]{Bate-geometric-Inventione}. Now we show that $H'$ is $G$-cr over $k$. Note that
\begin{equation*}
v(\sqrt a)^{-1}\cdot H' = \langle M, \epsilon_{18}(x)\epsilon_{-5}(\sqrt a x)\mid x\in \bar k\rangle.
\end{equation*}  
This shows that $v(\sqrt a)^{-1}\cdot H' \leq P_\lambda$. Thus $H'\leq v(\sqrt a)\cdot P_\lambda$. Using the argument in~\cite[Claim~3.6]{Uchiyama-Nonperfectopenproblem-D4} word-for-word, we have that $P_\lambda$ is not $k$-defined. (The same argument works since $v(\sqrt a)$ is not a $k$-point and $v(\sqrt a)\not\in P_\lambda$.) In the following we show that $v(\sqrt a)\cdot P_\lambda$ is the unique proper parabolic subgroup of $G$ containing $H'$, which implies that $H'$ is $G$-ir over $k$. 

Let $P_\mu$ be a proper parabolic subgroup containing $v(\sqrt a)^{-1}\cdot H'$. Then $M\leq P_\mu$. Since $M$ is $G$-cr, there exists a Levi subgroup $L$ of $P_\mu$ containing $M$. Since Levi subgroups of $P_\mu$ are $R_u(P_\mu)$-conjugate, we may assume $L=L_\mu$. Note that $L_\mu = C_G(\mu(\overline k^*))$, so $\mu(\overline k^*)$ must centralise $M$. Since $C_{G}(M)^{\circ}=G_{13}$, we have $\mu=g\cdot 13^{\vee}=g\cdot \lambda$ for some $g\in G_{13}$. Using the Bruhat decomposition, $g$ is of one of the following forms:
\begin{align*}
&\;g=\lambda(t)\epsilon_{13}(x_1)\textup{ or}\\
&\;g=\epsilon_{13}(x_1) n_{13}\lambda(t)\epsilon_{13}(x_2)\\
&\text{for some } x_1,x_2\in \overline{k} \text{ and } t\in \overline{k}^{*}.
\end{align*}
We rule out the second case. Suppose $g$ is of the second form. We have $\epsilon_{18}(1)\epsilon_{-5}(\sqrt a)\leq v(\sqrt a)\cdot H'\leq P_\mu= P_{g\cdot \lambda}=g\cdot P_{\lambda}$. So it is enough to show that $g^{-1}\cdot \epsilon_{18}(1)\epsilon_{-5}(\sqrt a)\not\in P_\lambda$. Since $U_{13}$ and $\lambda(\overline k^*)$ are contained in $P_{\lambda}$, we can assume $g=n_{13}$. A direct computation shows that $n_{13}^{-1}\cdot \epsilon_{18}(1)\epsilon_{-5}(\sqrt a)=\epsilon_{-23}(1)\epsilon_{-5}(\sqrt a)\not\in P_\lambda$. Thus $g$ is of the first form, but this implies $P_\mu = P_\lambda$. We are done.
\end{proof}

\begin{rem}
One might try to get another example of $H$ that is $G$-cr but not $G$-cr over $k$ (or a subgroup $H'$ that is $G$-cr over $k$ but not $G$-cr) by applying the special graph automorphism $\sigma$ of $F_4$ (in the sense of~\cite[Prop.~12.3.3]{Carter-simple-book}) on $H$ in Proposition~\ref{FirstMain} or on $H'$ in Proposition~\ref{SecondMain}. Remember that $\sigma(\epsilon_\zeta(x))=\epsilon_{\sigma(\zeta)}(x^{f(\zeta)})$ where $f(\zeta)=2$ if $\zeta$ is short and $f(\zeta)=1$ if $\zeta$ is long (we abuse the notation $\sigma$ for an automorphism of $\Psi(G)$). This method fails in both cases since $\sigma(v(\sqrt a))$ becomes a $k$-point. 
\end{rem}
\begin{rem}
The last remark gives a first counterexample to the following~\cite[Open Problem.~3.14]{Uchiyama-Nonperfect-pre} that asked: does the second part of Proposition~\ref{isogeny} hold without assuming $f$ central? We supply some details. We set $f=\sigma$. Then $\sigma$ is a (non-central) $k$-isogeny. We use $H$, $M$, and $L_\lambda$ in the proof of Proposition~\ref{FirstMain}. We have shown that $H$ is not $G$-cr over $k$. It is easy to see that $\sigma(M)$ is of type $G_2$ and is contained in $\sigma(L_\lambda)$ of type $C_3$. By~\cite[Table 9, ID 4]{Litterick-Thomas-ex-groups}, $\sigma(M)$ is $\sigma(L_\lambda)$-ir, thus $\sigma(L_\lambda)$-ir over $k$. Then $\sigma(M)$ is $G$-cr over $k$ and then, $\sigma(H)$ is $G$-cr over $k$ since $\sigma(H)$ is $G(k)$-conjugate to $\sigma(M)$. Since $\sigma^{-1}(\sigma(H))=H$, we are done. For an easy counterexample to the first part of Proposition~\ref{isogeny} without the centrality assumption (using Frobenius map), see~\cite[Ex.~3.12]{Uchiyama-Nonperfect-pre}. 
\end{rem}

%% file: TCC.tex
\section{Tits' centre conjecture}
In~\cite{Tits-colloq}, Tits conjectured the following:
\begin{con}\label{centerconjecture}
Let $X$ be a spherical building. Let $Y$ be a convex contractible simplicial subcomplex of $X$. If $H$ is a subgroup of the automorphism group of $X$ stabilizing $Y$, then there exists a simplex of $Y$ fixed by $H$.   
\end{con}
This so-called centre conjecture of Tits was proved by case-by-case analyses by Tits, M\"{u}hlherr, Leeb, and Ramos-Cuevas~\cite{Leeb-Ramos-TCC-GFA},~\cite{Muhlherr-Tits-TCC-JAlgebra},~\cite{Ramos-centerconj-Geo}. Recently, a uniform proof was given in~\cite{Weiss-center-Fourier}. In relation to the theory of complete reducibility, Serre showed~\cite{Serre-building}:
\begin{prop}\label{SerreContractible}
Let $G$ be a reductive $k$-group. Let $\Delta(G)$ be the building of $G$. If $H$ is not $G$-cr over $k$, then the convex fixed point subcomplex $\Delta(G)^H$ is contractible. 
\end{prop}
We identify the set of proper $k$-parabolic subgroups of $G$ with $\Delta(G)$ in the usual sense of Tits~\cite{Tits-book}. Note that for a subgroup $H$ of $G$, $N_{G(k)}(H)$ induces a group of automorphisms of $\Delta(G)$ stabilising $\Delta(G)^H$. Thus, combining the centre conjecture with Proposition~\ref{SerreContractible} we obtain
\begin{prop}\label{normalizer}
If a subgroup $H$ of $G$ is not $G$-cr over $k$, then there exists a proper $k$-parabolic subgroup of $G$ containing $H$ and $N_{G(k)}(H)$. 
\end{prop}
Proposition~\ref{normalizer} was an essential tool in proving various theoretical results on complete reducibility over nonperfect $k$ in~\cite{Uchiyama-Nonperfect-pre} and~\cite{Uchiyama-Nonperfectopenproblem-pre}. We have asked the following in~\cite[Rem.~6.5]{Uchiyama-Nonperfectopenproblem-pre}:
\begin{question}\label{centralizerQ}
If $H<G$ is not $G$-cr over $k$, then does there exist a proper $k$-parabolic subgroup of $G$ containing $HC_G(H)$?
\end{question}
The answer is affirmative if $C_G(H)$ is $k$-defined (or $k$ is perfect). In this case the set of $k$-points is dense in $C_G(H)$ (since we assume $k=k_s$) and the result follows from Proposition~\ref{normalizer}. The main result in this section is to present the first counterexample to Question~\ref{centralizerQ} when $k$ is nonperfect and $H$ is connected. (A counterexample with discrete $H$ was given in~\cite[Thm.~4.5]{Uchiyama-Nonperfectopenproblem-D4}.)
\begin{prop}\label{centralizerA}
Let $k$ be nonperfect of characteristic $2$. Let $G$ be simple of type $F_4$. Then there exists a nonabelian connected $k$-subgroup $H$ of $G$ such that $H$ is not $G$-cr over $k$ but $HC_G(H)$ is not contained in any proper $k$-parabolic subgroup of $G$. 
\end{prop}

\begin{proof}
We use the same $H$, $M$, $v(\sqrt a)$, and $\lambda$ as in the proof of Proposition~\ref{FirstMain}. We have shown that $H$ is not $G$-cr over $k$. We had $G_{13}\leq C_G(M)$, so $\langle M, G_{13}\rangle \leq M C_G(M)$. By the similar argument as in the proof of Proposition~\ref{SecondMain}, we can show that the unique proper parabolic subgroup of $G$ containing $\langle M, U_{13}\rangle $ is $P_\lambda$ (since $n_{13}\cdot 13 = -13$). It is clear that $P_\lambda$ does not contain $U_{-13}$. So there is no proper parabolic subgroup of $G$ containing $M C_G(M)$. Thus
there is no proper parabolic subgroup of $G$ containing $v(\sqrt a)\cdot (M C_G(M))=H C_G(H)$. 
\end{proof} 

\begin{rem}
Proposition~\ref{centralizerA} (and the proof) shows that it is hard to control $C_G(H)$ when $C_G(H)$ is not $k$-defined even if $H$ is connected. This makes Open Problem~\ref{centraliserProblem} difficult even for connected $H$.  
\end{rem}

%% file: acknowledgements.tex
\section*{Acknowledgements}
While undertaking the work for this article, the second and third authors were supported by Alexander von Humboldt Fellowships. The third author also acknowledges the financial support of JSPS Grant-in-Aid for Early-Career Scientists (19K14516). The authors would like to thank Gerhard R\"{o}hrle and Michael Bate for helpful comments. 